\documentclass{imsart}
\RequirePackage[OT1]{fontenc}
\RequirePackage{amssymb, amsfonts, amstext,amsthm, latexsym,amsmath}
\RequirePackage{graphicx, bm, bbm, yfonts, colortbl}
\RequirePackage[utf8]{inputenc}
\RequirePackage[numbers]{natbib}
\RequirePackage[colorlinks,citecolor=blue,urlcolor=blue]{hyperref}
\startlocaldefs
\numberwithin{equation}{section}
\numberwithin{subsection}{section}
\theoremstyle{plain}

\newtheorem{teo}{Theorem}[section]
\newtheorem{prop}{Proposition}[section]
\newtheorem{cor}{Corollary}[section]
\newtheorem{lem}{Lemma}[section]
\newtheorem{rem}{Remark}[section]

\newcommand{\beq}{\begin{equation}}
\newcommand{\beqn}{\begin{equation*}}
\newcommand{\eeq}{\end{equation}}
\newcommand{\eeqn}{\end{equation*}}
\newcommand{\N}{{\mathbb N}}
\newcommand{\R}{\mathbb{R}}
\newcommand{\Z}{{\mathbb Z}}

\endlocaldefs

\begin{document}

\begin{frontmatter}
\title{A linear stochastic biharmonic heat equation: hitting probabilities}
\runtitle{stochastic biharmonic heat equation}
%\thankstext{T1}{Footnote to the title with the `thankstext' command.}

\begin{aug}
\author{\fnms{Adri\'an} \snm{Hinojosa-Calleja}\thanksref{t3}\ead[label=e1]{ahinojosa@ub.edu}}
\and
\author{\fnms{Marta} \snm{Sanz-Sol\'e}\thanksref{t2}\ead[label=e2]{marta.sanz@ub.edu}}

\address{Facultat de Matem\`atiques i Inform\`atica, Universitat de Barcelona\\
Gran Via de les Corts Catalanes, 585, E-08007 Barcelona, Spain\\
Barcelona Graduate School of Mathematics\\
\printead{e1,e2}}

%\author{\fnms{Third} \snm{Author}
%\ead[label=e3]{third@somewhere.com}
%\ead[label=u1,url]{www.foo.com}}

%\address{Address of the Third author\\
%usually few lines long\\
%usually few lines long\\
%\printead{e3}\\
%\printead{u1}}

%\thankstext{t1}{Some comment}
\thankstext{t2}{Partially supported by the grant PID2020-118339GB-I00 from the \textit{Ministerio de Ciencia e Innovaci\'on}, Spain.}
\thankstext{t3}{Supported by the grant BES-2016077051 from the \textit{Ministerio de Ciencia e Innovaci\'on}, Spain.}
\runauthor{A. Hinojosa-Calleja, M. Sanz-Sol\'e}

\affiliation{University of Barcelona}

\end{aug}

\begin{abstract}
Consider the linear stochastic biharmonic heat equation on a $d$--dimensional torus ($d=1,2,3$), driven by a space-time white noise and with periodic boundary conditions:  
\beq
\label{0}
\left(\frac{\partial}{\partial t}+(-\Delta)^2\right) v(t,x)= \sigma \dot W(t,x),\  (t,x)\in(0,T]\times \mathbb{T}^d,
\eeq
 $v(0,x)=v_0(x)$.
  We find the canonical pseudo-distance corresponding to the random field solution, therefore the precise description of the anisotropies of the process. We see that for $d=2$, they include a $z(\log \tfrac{c}{z})^{1/2}$ term. Consider $D$ independent copies of the random field solution to \eqref{0}. Applying the criteria proved in \cite{hin:san}, we establish upper and lower bounds for the probabilities that the path process hits bounded Borel sets.This yields results on the polarity of sets and on the Hausdorff dimension of the path process. 

\end{abstract}


%\tableofcontents
\end{frontmatter}
%%%%%%%%% Introduction
%%%%%%%%%
%%%%%%%%%

\section{Introduction}
\label{s1}

This paper is motivated by the study of sample path properties of stochastic partial differential equations (SPDEs) and its applications to questions like the polarity of sets for the path process and its Hausdorff dimension (a.s.). We focus on a system of stochastic linear biharmonic heat equations on a $d$-dimensional torus, $d=1,2,3$,  with periodic boundary conditions (see \eqref{eq1.1}). This SPDE is the linearization at zero of a Cahn-Hilliard equation with a space-time white noise forcing term. 

In the last two decades, there has been many contributions to the subject of this paper. A large part of them concern Gaussian random fields, the case addressed in this work. A representative sample of results can be found in \cite{dal:kho}, \cite{dal:san}, \cite{x-m-d}, \cite{ss-viles}, \cite{xia}, and references therein.
Central to the study is obtaining upper and lower bounds on the probabilities that the random field hits a Borel set $A$, in terms of the Hausdorff measure and the capacity, respectively, of $A$.  In the derivation of the bounds --named {\em criteria for hitting probabilities}-- a major role is played by the canonical pseudo-distance associated to the process. For a random field $(v(t,x),\ (t,x)\in[0,T]\times \mathbb{D})$, $\mathbb{D}\subset \R^d$, this notion is defined by
\begin{equation*}
\mathfrak{d}_{v}((t,x),(s,y))=\Vert v(t,x)-v(s,y)\Vert_{L^2(\Omega)},\ (t,x), (s,y)\in[0,T]\times \mathbb{D}.
\end{equation*}
When $\mathfrak{d}_{v}((t,x),(s,y))$ compares, up to multiplicative constants, with  $\vert t-s\vert^{\alpha_0}+\sum_{j=1}^d\vert x_j-y_j\vert^{\alpha_j}$, $\alpha_0, \alpha_j\in (0,1)$, \cite{xia}[Theorem 7.6] and \cite{dal:san}[Theorems 2.1, 2.4 and 2.6]) provide useful criteria for hitting probabilities. 

Let $(u(t,x),\ (t,x)\in[0,T]\times\mathbb{T}^d)$, $d=1,2,3$, be the random field solution to the biharmonic heat equation driven by space-time noise, given in Theorem \ref{t1.1}. 
We prove in Theorem \ref{teo2.1}  that the associated canonical pseudo-distance 
$\mathfrak{d}_{u}((t,x),(s,y))$  compares with
 \beqn
  \left(\vert t-s\vert^{1-d/4}+\left(\log\frac{C(d)}{\vert x-y\vert}\right)^\beta\vert x-y\vert^{2\wedge(4-d)}\right)^{\frac{1}{2}}, \ \beta=1_{\{d=2\}}.
  \eeqn
Thus, when $d=2$ this example does not fall into the range of applications of the criteria cited above. 

In \cite{hin:san}[Theorems 3.2, 3.3, 3.4, 3.5], we proved extensions of \cite{xia}[Theorem 7.6] to cover cases where the canonical pseudo-distance has anisotropies described by gauge functions other than power functions. This was initially motivated by the study of a linear heat equation with fractional noise
(see \cite{hin:san}[Section 4]). From the above discussion, we see that the biharmonic heat equation provides a new case of application of such extended criteria.

The structure and contents of the paper are as follows. Section \ref{s2} is about preliminaries. We formulate and prove the existence of a random field solution to the biharmonic heat equation, and recall the notions of  Hausdorff measure relative to a gauge function and capacity relative to a symmetric potential. Section \ref{s3} is devoted to find the equivalent pseudo-distance for the canonical metric --a result of independent interest. The proof relies on a careful analytical study of the Green's function of the biharmonic operator $\mathcal{L} = \frac{\partial}{\partial t} +(-\Delta)^2$ on $(0,T)\times \mathbb{T}^d$. 
With this fundamental result at hand and some additional properties of $(u(t,x))$ proved in Sections \ref{s4} and \ref{s-6}, we are in a position to apply Theorems 3.4 and 3.5 of \cite{hin:san}. We deduce Theorem \ref{s5-t5.1} on upper and lower bounds for the hitting probabilities of $D$-dimensional random vectors consisting of independent copies of $(u(t,x))$.  These are in terms of the $\bar g_q$-Hausdorff measure and the $(\bar g_q)^{-1}$-capacity, respectively, with $\bar g_q$ defined in \eqref{5.1}. Notice that for $d=1,3$, the bounds are given by the classical Hausdorff mesure and the Bessel-Riesz capacity, respectively. In the second part of Section \ref{s5}, we highlight some consequences of Theorem \ref{s5-t5.1} on polarity of sets and Hausdorff dimension of the path process. 
The application of Theorems 3.4 and 3.5 of \cite{hin:san} imposes the restriction
$D>D_0$, where $D_0=[(4-d)/8]^{-1}+d[1\wedge(2-d/2)]^{-1}$. We also discuss the case $D<D_0$ and present some conjectures concerning the critical case $D=D_0$ in the last part of Section \ref{s5}.

%%%%%
%%%%%%%%% Section 2
%%%%%%%%%
\section{Notations and preliminaries}
\label{s2}

We introduce some notation used throughout the paper. As usually, $\mathbb{N}$ denotes the set of natural numbers $\{ 0,1,2,...\}$; we set  $\mathbb{Z}_2 = \{ 0,1\}$, and for any integer $d\ge 1$, $\mathbb{N}^{d,*} = (\mathbb{N}\setminus\{ 0\})^d$. For any multiindex $k=(k_1,\ldots,k_d)\in \N^d$, we set $|k|=(\sum_{j=1}^d k_j^2)^{1/2}$, and denote by $n(k)$ the number of null components of $k$. 

Let $\mathbb{S}^1$ be the circle and $\mathbb{T}^d= \mathbb{S}^1\times \stackrel{d}{\ldots}\times \mathbb{S}^1$ the $d$-dimensional torus. 
For $x\in \mathbb{T}^d$, $|x|$ denotes the Euclidean norm. If we identify  
$\mathbb{T}^d$ with the periodic cube $[-\pi,\pi]^d$, meaning that opposite sides coincide, $|x|$ can be interpreted as the distance of $x$ to the origin.

%Using the representation $X_j= e^{i x_j}$, $x_j\in[0,2\pi)$, for points $X_j\in \mathbb{S}^1$, we identify $X_j$ with $x_j$ and define the norm on $\mathbb{T}^d$ by 
%\beqn
%\vert X\vert_{\mathbb{T}^d}^2 :=\vert x\vert^2 =\sum_{j=1}^d x_j ^2, \ X= (e^{i x_j})_{1\le j\le d},\ x=(x_j)_{1\le j\le d}.
%\eeqn
%This amounts to identify $\mathbb{T}^d$ with the periodic box $[0,2\pi]^d$.

For $x\in[0,2\pi)$, let $\varepsilon_{0,k}(x)=\pi^{-1/2}\sin(kx)$, $\varepsilon_{1,k}(x)=\pi^{-1/2}\cos(kx)$, $k\in\mathbb{N}^\ast$,
and $\varepsilon_{1,0}(x) = (2\pi)^{-1/2}$. The set of functions $\textbf{B}$ defined on $\mathbb{T}^d$ consisting of
\beqn
\varepsilon_{i,k}:=\varepsilon_{i_1,k_1}\otimes \cdots\otimes \varepsilon_{i_d,k_d},\ i=(i_1,\ldots,i_d)\in \mathbb{Z}_2^d,
\eeqn
with $k_j\in \mathbb{N}^*$ if $i_j=0$, and $k_j\in \N$ if $i_j=1$, is an orthonormal basis for $L^2(\mathbb{T}^d)$. 

%With the aim of easing the notation at several places, we define $\varepsilon_{0,0}(x)=0$.
Define 
\beqn
(\mathbb{Z}_2\times \mathbb{N})^d_{+} = \{(i,k)\in (\mathbb{Z}_2\times \mathbb{N})^d: (i_j,k_j)\ne (0,0),\  \forall j=1,\ldots,d\}.
\eeqn
Notice that  $\textbf{B}= \{\varepsilon_{i,k} =\varepsilon_{i_1,k_1}\otimes \cdots\otimes \varepsilon_{i_d,k_d},\ (i,k)\in(\mathbb{Z}_2\times \mathbb{N})^d_{+}\}$. 

The following equality is a straightforward consequence of the formula for the cosinus of a sum of angles: For any $x,y\in\mathbb{T}^d$,
\begin{equation}\label{basis}
\sum_{i\in\mathbb{Z}^d_2}\varepsilon_{i,k}(x)\varepsilon_{i,k}(y)=\frac{1}{2^{n(k)}\pi^d}\prod_{j=1}^d\cos(k_j (x_j-y_j) ),\ 
k\in\mathbb{N}^{d}\ {\text{with}}\ (i,k)\in (\mathbb{Z}_2\times \mathbb{N})^d_{+}.
\end{equation}
%where $n(k)$ denotes the number of null coordinates of $k$ and, by convention $\varepsilon_{0,0}(x)=0$.

Let $(-\Delta)^2$ be the biharmonic operator (also called {the \em bilaplacian}) on  $L^2(\mathbb{T}^d)$.  The basis $\textbf{B}$ is a set of eigenfunctions of $(-\Delta)^2$ with associated eigenvalues $\lambda_k=\sum_{j=1}^d k_j^4$,  $k\in \mathbb{N}^{d}$. Observe that $d^{-1}|k|^4\le \lambda_k\le  |k|^4$, and $\inf_{k\in \N^{d,*}} \lambda_k=d$.

The Green's function of the biharmonic heat operator $\mathcal{L} = \frac{\partial}{\partial t} +(-\Delta)^2$ on $(0,T] \times \mathbb{T}^d$ is given by
\begin{align}
\label{eq1.2}
G(t;x,y)=\sum_{(i,k)\in(\mathbb{Z}_2\times \mathbb{N})^d_{+} } e^{-\lambda_k t}\varepsilon_{i,k}(x)\varepsilon_{i,k}(y)
=\sum_{k\in\mathbb{N}^d}\frac{e^{-\lambda_k t}}{2^{n(k)}\pi^d}\prod_{j=1}^d\cos(k_j (x_j-y_j) ),% \quad t\in(0,T],\ x,y\in \mathbb{T}^d.
\end{align}
the last equality being a consequence of \eqref{basis}.

This paper concerns the linear stochastic biharmonic heat equation 
\beq
\label{eq1.1}
\begin{cases}
\left(\frac{\partial}{\partial t}+(-\Delta)^2\right) v(t,x)= \sigma \dot W(t,x), & (t,x)\in(0,T]\times \mathbb{T}^d,\\
 v(0,x)=v_0(x),
 \end{cases}
\eeq
where $(\dot W(t,x))$ is a space-time white noise on $[0,T]\times \mathbb{T}^d$, $\sigma\in\R\setminus \{0\}$ and $v_0: \  \mathbb{T}^d\longrightarrow \R$.

We consider the random field solution to \eqref{eq1.1}, that is, the stochastic process
\begin{equation}
\label{eq1.1bis}
v(t,x)=\int_{\mathbb{T}^d} G(t;x,z)v_0(z)dz + \sigma \int_0^t\int_{\mathbb{T}^d} G(t-r;x,z)W(dr,dz),
\end{equation}
with $G$ given in \eqref{eq1.2}, and the stochastic integral is a Wiener integral  with respect to space-time white noise.

We assume that, for any $(t,x)\in(0,T]\times \mathbb{T}^d$, the function $\mathbb{T}^d\ni z\mapsto G(t;x,z) v_0(z)$ belongs to $L^1(\mathbb{T}^d)$. Along with the next Theorem, this yields that $(v(t,x),\ (t,x)\in [0,T]\times \mathbb{T}^d)$ is a well-defined Gaussian process.

\begin{teo}
\label{t1.1} Let
\beqn
u(t,x) = \int_0^t\int_{\mathbb{T}^d} G(t-r;x,z)W(dr,dz),\quad (t,x)\in [0,T]\times \mathbb{T}^d.
\eeqn
The stochastic process $(u(t,x),\ (t,x)\in[0,T]\times\mathbb{T}^d)$ is well-defined if and only if $d=1, 2, 3$. In this case,
\beq
\label{1.1}
\sup_{(t,x)\in[0,T]\times \mathbb{T}^d}E(\vert u(t,x)\vert^2)<\infty.
\eeq
\end{teo}
\begin{proof} 
Fix $(t,x)\in(0,T]\times \mathbb{T}^d$. By \eqref{eq1.2} and applying Fubini's theorem, we have
\begin{align}
\label{eq1.4}
\int_0^t dr \int_{\mathbb{T}^d} dz\  G^2(t-r; x,z) &= \sum_{(i,k)\in(\mathbb{Z}_2\times \mathbb{N})^d_{+} }\varepsilon_{i,k}^2(x) \left(\int_0^t dr\  e^{-2\lambda_k r}\right)\notag\\
&=\sum_{k\in\mathbb{N}^{d}}\frac{1}{2^{n(k)}\pi^d}\int_0^t dr\  e^{-2\lambda_k r}\notag\\
& = \frac{t}{(2\pi)^d} + \sum_{\substack{k\in\mathbb{N}^{d}\\0\le n(k)\le d-1}} \frac{1-e^{-2\lambda_k t}}{2^{n(k)+1}\pi^d\lambda_k}.
\end{align}
%{\substack{k\in  \mathbb{N}^{d,*}\\ |k| > r}} 
 Apply the inequalities
$\frac{u}{1+u}\leq 1-e^{-u}\leq 1$, valid for all $u\geq 0$, to see that the series in \eqref{eq1.4} is equivalent to a harmonic series $\sum_{\substack{k\in\mathbb{N}^{d}\\ 0\le n(k)\le d-1}}\frac{1}{|k|^4}$, which converges if and only if $d\le 3$. Equivalently, the Wiener integral defining $u(t,x)$ is well-defined if and only if $d\le 3$.
This finishes the proof of the first statement.

By the isometry property of the Wiener integral, $E((u(t,x))^2)$ is equal to the right-hand side of \eqref{eq1.4}. Taking the supremum in \eqref{eq1.4}, we have
\begin{align*}
\sup_{(t,x)\in[0,T]\times \mathbb{T}^d} E((u(t,x))^2)
&\le \frac{T}{(2\pi)^d}+ \sup_{t\in[0,T]} \sum_{k\in\mathbb{N}^{d}}\frac{1-e^{-2\lambda_k t}}{2^{n(k)+1}\pi^d\lambda_k}\\
& \le \frac{T}{(2\pi)^d} + \sum_{k\in\mathbb{N}^{d}}\frac{1}{2^{n(k)+1}\pi^d\lambda_k} \le C(T,d).
\end{align*}
\end{proof}
In the sequel we will take $d\in\{1,2,3\}$. 
\medskip

In the last part of this section, we recall the notions of Hausdorff measure and capacity that will be used 
in of Section \ref{s5}.
\smallskip

\noindent{\em $g$-Hausdorff measure}
\smallskip

Let $\varepsilon_0>0$ and $g: [0,\varepsilon_0 ]\rightarrow \R_+$ be a continuous strictly increasing function satisfying $g(0)=0$. The {\em  $g$-Hausdorff measure} of 
 a Borel set $A\subset\mathbb{R}^D$ is defined  by 
\beqn
\mathcal{H}_g(A)=\lim_{\varepsilon\downarrow 0} \inf\left\{\sum_{i=1}^\infty g(2r_i): A\subset\bigcup_{i=1}^\infty B_{r_i}(x_i),\ \sup_{i\geq 1}r_i\leq\varepsilon\right\}
\eeqn
(see e.g. \cite{rogers}). In this paper, we will use this notion referred to two examples:\ (i)  $g(\tau)= \tau^\gamma$, with $\gamma>0$; this is the classical $\gamma$-dimensional Hausdorff measure.  (ii)\  $g(\tau) = \tau^{\nu_1}\left(q^{-1}(\tau)\right)^{-\eta}$, with $q(\tau) = \tau^{\nu_2}\left(\log\frac{c}{\tau}\right)^{\delta}$, 
$\nu_1, \nu_2, \eta, \delta >0$.

By coherence with the definition of the $\gamma$-dimensional Hausdorff measure when $\gamma<0$, if $g(0)=\infty$, we set $\mathcal{H}_g(A)=\infty$.
\smallskip

\noindent{\em Capacity relative to a symmetric potential kernel}
\smallskip

Let $\textgoth{g}: \R^D\longrightarrow \R_+\cup \{\infty\}$ be continuous on $\R^D\setminus\{0\}$, symmetric, $\textgoth{g}(z)> 0$, for all $z\ne 0$,  $\textgoth{g}(0) = \infty$. This function is called a {\em symmetric potential}. The $\textgoth{g}$-{\em capacity} of a Borel set $A\subset\mathbb{R}^D$ is defined  by
\beqn
\text{Cap}_{\textgoth{g}}(A) = \left[ \inf_{\mu\in\mathbb{P}(A)}\mathcal{E}_{\textgoth{g}}(\mu)\right]^{-1},
\eeqn
where 
$\mathcal{\mathcal{E}}_{\textgoth{g}}(\mu) = \int_{\R^D\times \R^D} \textgoth{g}(y-\bar y)\ \mu(dy)\mu(d\bar y)$ and
$\mathbb{P}(A)$ denotes the set of probability measures on $A$. 
If $\textgoth{g}(0)\in [0,\infty)$, we set $\text{Cap}_{\textgoth{g}}(A)= 1$, by convention. 

In this article, we will use this notion with $\textgoth{g}=1/g$, where $g$ is as in the examples (i) and (ii) above.  Observe that, in the example (i), the $\textgoth{g}$-{\em capacity} is  the Bessel-Riesz capacity, usually denoted by $\text{Cap}_\gamma (A)$ (see e.g. \cite[p. 376]{kho}).
\smallskip

Throughout the article, 
positive real constants are denoted by $C$, or variants, like $\bar C$, $\tilde C$, $c$, etc. If we want to make explicit the dependence on some parameters $a_1, a_2,\ldots$, we write $C(a_1, a_2, \ldots)$ or $C_{a_1,a_2,\ldots}$. When writing $\log\left(\frac{C}{z}\right)$, we will assume that $C$ is large enough to ensure $\log\left(\frac{C}{z}\right)\ge1$.

%%%%%%%

%%%%%%
%%%%%% Section 3
%%%%%%
\section{Equivalence for the canonical metric} 
\label{s3}

For the process $u$ of Theorem \ref{t1.1}, we define 
\begin{equation}
\label{d}
\mathfrak{d}_{u}((t,x),(s,y))=\Vert u(t,x)-u(s,y)\Vert_{L^2(\Omega)}.
\end{equation}
This is the canonical pseudo-distance associated with $u$. 
This section is devoted to establish an equivalent (anisotropic) pseudo-distance for $\mathfrak{d}_{u}$.

 Throughout the proofs, we will make frequent use of  the identity
\begin{align}
\label{isometry}
&\Vert u(t,x)-u(s,y)\Vert^2_{L^2(\Omega)}
 = \frac{1}{2^{n(k)+1}\pi^d}\notag\\
&\quad\times \sum_{k\in\N^{d,*}} \frac{1-e^{-2 \lambda_k s}}{\lambda_k}
\left(e^{-2\lambda_k(t-s)} + 1 - 2 e^{-\lambda_k(t-s)}\prod_{j=1}^d\cos(k_j (x_j-y_j) )\right)\notag\\
&\quad +\frac{1}{2^{n(k)+1}\pi^d}\sum_{k\in\N^{d,*}} \frac{1-e^{-2\lambda_k (t-s)}}{\lambda_k}
+\frac{t-s}{(2\pi)^d},
\end{align}
$0\le s\le t$.
This formula is proved using the Wiener isometry 
\begin{align}
\label{isometry-easy}
&\Vert u(t,x)-u(s,y)\Vert^2_{L^2(\Omega)} = \int_0^t dr \int_{\mathbb{T}^d} dz\ (G(t-r; x,z) - G(s-r; y,z))^2\notag\\
&\ = \int_0^s dr \int_{\mathbb{T}^d} dz\ (G(t-r; x,z) - G(s-r; y,z))^2 + \int_s^t dr \int_{\mathbb{T}^d} dz\ G^2(t-r; x,z),
\end{align}
(the last equality holds because the Green's function $G(r;y,z)$ vanishes if $r<0$) and  using the definition \eqref{eq1.2}.
The first (respectively, second) series term in \eqref{isometry} equals the first (respectively second) integral on the rignt-hand side of \eqref{isometry-easy}.

%%%%%Increments in time
%%%%%Increments
We start by analyzing the $L^2(\Omega)$-increments in the time variable of the process $(u(t,x))$.
\begin{prop}
\label{p3.1}
1.\ There exists constants $c_1(d,T)$ and $c_2(d)$ such that, for all  $s, t\in[0,T]$, $x\in \mathbb{T}^d$,
\beq
\label{tbd}
c_1(d,T)\vert t-s\vert^{1-d/4}\leq \Vert u(t,x)-u(s,x)\Vert^2_{L^2(\Omega)} \leq c_2(d)\vert t-s\vert^{1-d/4}.
\eeq
\noindent{2.}\ For any $(t,x), (s,y)\in [0,T]\times \mathbb{T}^d$,
\beq
\label{tbd-bis}
c_1(d,T)\vert t-s\vert^{1-d/4}\leq \Vert u(t,x)-u(s,y)\Vert^2_{L^2(\Omega)}, 
\eeq
where $c_1(d,T)$ is the same constant as in \eqref{tbd}.
\end{prop}
%%%%Proof
\begin{proof}  Without loss of generality, we suppose $0\le s< t\le T$. 
\smallskip

%By the Wiener isometry,
%\beq
%\label{wi}
%\Vert u(t,x)-u(s,x)\Vert^2_{L^2(\Omega)} = \int_0^t dr\int_{\mathbb{T}^d} dz\ \left(G(t-r;x,z) - G(s-r;x,z)\right)^2.
%\eeq
\noindent
Use the first equality in \eqref{isometry-easy} and then apply Lemma \ref{s-a-l.1} with $h:=t-s$. This yields the second inequality in \eqref{tbd}.
\smallskip

\noindent From \eqref{isometry}, we have
\begin{align}
\label{w1p}
\Vert u(t,x)-u(s,x)\Vert^2_{L^2(\Omega)}\ge \frac{1}{2^{n(k)+1}\pi^d} \sum_{k\in\mathbb{N}^{d,*}} \frac{1-e^{-2\lambda_k(t-s)}}{\lambda_k}. 
\end{align}

Let $r\ge d$. Applying the inequality $1-e^{-u}\geq \frac{u}{1+u}$, $u\geq 0$, we obtain
\begin{align*}
 &\sum_{k\in\mathbb{N}^{d,*}} \frac{1-e^{-2 \lambda_k(t-s)}}{\lambda_k} \ge 2(t-s) \sum_{\substack{k\in  \mathbb{N}^{d,*}\\ |k| > r} } \frac{1}{1+2 \lambda_k(t-s)}\\
 & \qquad \ge \frac{2(t-s)}{r^{-4} + 2(t-s)}\sum_{\substack{k\in  \mathbb{N}^{d,*}\\ |k| > r}} \frac{1}{|k|^4}
 = C_d\ \frac{2(t-s)}{r^{-4} + 2(t-s)} r^{d-4},
 \end{align*}
 since $\lambda_k\le |k|^4$.
 Choosing $r=\left(\tfrac{d^4 T}{t-s}\right)^{1/4}$, the inequality above yields
 \beqn
 \Vert u(t,x)-u(s,x)\Vert^2_{L^2(\Omega)} \ge c_1(d,T)(t-s)^{1-d/4},
 \eeqn
 with $ c_1(d,T) = C_d \frac{2d^d T^{d/4}}{1+2d^4 T}$. This is the lower bound in \eqref{tbd}.
 \smallskip
 
 Notice that from \eqref{isometry} we deduce
 \beqn
\Vert u(t,x)-u(s,y)\Vert^2_{L^2(\Omega)}\ge \frac{1}{2^{n(k)+1}\pi^d} \sum_{k\in\mathbb{N}^{d,*}} \frac{1-e^{-2\lambda_k(t-s)}}{\lambda_k}. 
\eeqn
Hence the proof above yields \eqref{tbd-bis}.
 \end{proof}
 
 %%%%%%
 %%%%%%Increments in space
For any $j=1,\ldots,d$, fix real numbers $0<c_{0,j}<2\pi$ and define $J_j=[c_{0,j},2\pi-c_{0,j}]$ and $J= J_1\times \ldots \times J_d\subsetneq\mathbb{T}^d$. The next statement deals with increments in space. 

\begin{prop}
\label{prop2.2}  Let $(u(t,x),\ (t,x)\in [0,T]\times \mathbb{T}^d)$ be the stochastic process defined in Theorem \ref{t1.1} and let $J$ be a compact set as described before. There exists positive constants $c(d)$, $C(d)$, $c_3(d)$ and $c_4(d)$ such that, for any $t>0$, $x, y \in J$,
%$z=(e^{ix_j})_{1\le j\le d}$, $\bar z=(e^{iy_j})_{1\le j\le d}$, 
\begin{align}\label{sbd}
&c_3(d) C_t \left(\log\frac{c(d)}{\vert x-y\vert}\right)^\beta\vert x-y\vert^{2\wedge(4-d)}\notag\\
&\quad  \qquad \le \Vert u(t,x)-u(t,y)\Vert^2_{L^2(\Omega)} \leq c_4(d)\left(\log\frac{C(d)}{\vert x-y\vert}\right)^\beta\vert x-y\vert^{2\wedge(4-d)},
\end{align}
where $C_t = (1-e^{-2dt})$, and $\beta= 1_{\{d=2\}}$.

The upper bound holds for any $(t,x)\in[0,T]\times \mathbb{T}^d$. The lower bound holds for any $x,y\in\mathbb{T}^d$ if $|x-y|$ is small enough. For $t=0$, the lower bound is non informative.
\end{prop}
\begin{proof}
%To simplify notation we write $\vert \ast \vert$ for $\vert \ast\vert_{\mathbb{T}^d}$.

\noindent{\em Upper bound.} \  From \eqref{isometry} we deduce
\begin{align}
\label{basic1}
 \Vert u(t,x)-u(t,y)\Vert^2_{L^2(\Omega)} &=  \frac{1}{2^{n(k)}\pi^d} \sum_{k\in\N^{d,*}} \frac{1-e^{-2\lambda_k t}}{\lambda_k}
\left(1 - \prod_{j=1}^d\cos(k_j (x_j-y_j) )\right).%notag\\
%&\frac{1}{2^{n(k)}\pi^d} \sum_{\substack{k\in\mathbb{N}^{d}\\ 0\le n(k)\le d-1}}\frac{1-e^{-2\lambda_k t}}{\lambda_k}
%\left(1 - \prod_{j=1}^d\cos(k_j (x_j-y_j) )\right)
%&\le C(d) \sum_{k\in \mathbb{N}^{d}}\frac{1\wedge(\vert k\vert\vert x-y\vert)^2}{\vert k\vert^4}.
\end{align}
Observe also that, because of \eqref{basis}, 
\beq
\label{meanvalue}
1-\prod_{j=1}^d \cos(k_j(x_j-y_j)) = 2^{n(k)-1}\pi^d \sum_{i\in\Z_2^d} (\varepsilon_{i,k}(x) - \varepsilon_{i,k}(y))^2\le \bar C(d)(1\wedge(\vert k\vert\  \vert x-y\vert)^2).
\eeq
 for any $(i,k)\in(\Z_2\times \N)^d$.

\noindent{\em Case $d=1$.}\  Since $\sum_{k\ge 1}\frac{1}{k^2}< \infty$, from \eqref{basic1} and \eqref{meanvalue} we have
\beq
\label{d=1}
\Vert u(t,x)-u(t,y)\Vert^2_{L^2(\Omega)}= \frac{1}{\pi}\sum_{k\ge 1} \frac{1-e^{-2\lambda_kt}}{\lambda_k}(1-\cos(k(x-y))) 
\le C_d |x-y|^2.
\eeq
\smallskip

\noindent{\em Case $d=2,3$}.\ For any $k\in \N^{d}$, let $I_k=[k_1,k_1+1)\times \cdots\times[k_d,k_d+1)$. Observe that for any $d$-dimensional vector $z\in I_k$, 
we have $|z| \le |k| + \sqrt d$. Fix $\rho_0\ge \lfloor 3\sqrt d\rfloor + 1$ and let $\alpha>0$. Then,
%\begin{align}
%\label{away-alternative}
% \sum_{\substack{k\in  \mathbb{N}^{d,*}\\ |k| > r}} \frac{1}{|k|^\alpha} & \le  \sum_{\substack{k\in  \mathbb{N}^{d,*}\\ |k| > r}}\frac{1}{|k|_E^\alpha} 
% \le  \sum_{\substack{k\in  \mathbb{N}^{d,*}\\ |k| > r}} \int_{I_{k-1}} \frac{dz}{|z|_E^\alpha}
%=  \sum_{\substack{k\in  \mathbb{N}^{d,*}\\ |k| > r-d}} \int_{I_{k}} \frac{dz}{|z|_E^\alpha}\notag\\
%& \le C_d \int_{r-d}^\infty \rho^{-\alpha+d-1}. 
% \end{align}
  \begin{align}
\label{away}
T_1(\alpha,\rho_0)&:=\sum_{\substack{k\in  \mathbb{N}^{d}\\ |k| \ge \rho_0}} \frac{1}{|k|^\alpha}  \le  \sum_{\substack{k\in  \mathbb{N}^{d}\\ |k| \ge \rho_0}} \int_{I_k} \frac{dz}{(|z|-\sqrt d)^\alpha}
 \le  C_d\int_{\rho_0}^\infty \frac{\rho^{d-1}\ d\rho}{(\rho-\sqrt d)^\alpha}\notag\\
&\le  C_{d,\alpha} \int_{\rho_0}^\infty \rho^{d-1-\alpha}\ d\rho,
 \end{align}
where the last inequality holds because on $[\rho_0,\infty)$,  $\rho-\sqrt d\ge 1/2\rho$. 

Let $\rho_0$ be as above, $\rho_1 = \left\lfloor(3/2)\sqrt d\right\rfloor + 1$, and  $\beta>0$. 
By arguments similar to those used to obtain \eqref{away}, we deduce
\begin{align}
\label{close}
T_2(\beta,\rho_0)&= \sum_{\substack{k\in  \mathbb{N}^{d}\\ \rho_1\le |k| <\rho_0}} \frac{1}{|k|^\beta}   \le  \sum_{\substack{k\in  \mathbb{N}^{d}\\ \rho_1\le |k| <\rho_0}} \int_{I_k} \frac{dz}{(|z|-\sqrt d)^\beta}
\le  C_d\int_{\rho_1}^{\rho_0}\frac{\rho^{d-1}\ d\rho}{(\rho-\sqrt d)^\beta}\notag\\
& \le  C_{d,\beta} \int_{\rho_1}^{\rho_0} \rho^{d-1-\beta}\ d\rho, 
 \end{align}
 where in the last inequality, we have used that on $[\rho_1,\rho_0]$, $\rho-\sqrt d\ge (1/5)\rho$.
 
 Set $h= |x-y|$ and $\rho_0 = \left\lfloor c_d h^{-\frac{2\wedge (4-d)}{4-d}}\right\rfloor +1$, where $c_d= 3\sqrt d (2\pi\sqrt d)^{\frac{2\wedge (4-d)}{4-d}}$. Notice that $\rho_0\ge \lfloor 3\sqrt d\rfloor + 1$. Then, from \eqref{basic1} we have
 \beq
 \label{join}
  \Vert u(t,x)-u(t,y)\Vert^2_{L^2(\Omega)} \le C(d)\left[T_1(4,\rho_0)+ h^2 \left(T_2(2,\rho_0) +  \sum_{\substack{k\in  \mathbb{N}^{d}\\ 1\le |k| <\rho_1}} \frac{1}{|k|^2}\right)\right].
  \eeq 
 Using \eqref{away}, with the choice of $\rho_0$ specified above, we see that $T_1(4,\rho_0)\le C_d h^{2\wedge(4-d)}$ and 
 \beqn
 T_2(2,\rho_0) \le C_d\times
 \begin{cases}
 \log\left(\frac{C}{h}\right),& \ d=2,\\
 h^{-1},&\ d=3.
 \end{cases}
 \eeqn 
 Since $ \sum_{k\in  \mathbb{N}^{d},\ 1\le |k| <\rho_1} \frac{1}{|k|^2}=\tilde c_d <\infty$, substituting the above estimates in the right-hand side of \eqref{join} we obtain the  upper bound in \eqref{sbd}.
\smallskip

\noindent{\em Lower bound.\  Case $|x-y|$ small.}\ We start from \eqref{basic1} to obtain
\beq
\label{lower1}
 \Vert u(t,x)-u(t,y)\Vert^2_{L^2(\Omega)} \ge   \frac{1-e^{-2t}}{ 2^{n(k)}\pi^d} \sum_{k\in\N^{d,*}} \frac{1 - \prod_{j=1}^d\cos(k_j (x_j-y_j) )}{|k|^4}.
\eeq
Let $T(x,y)$ denote the series on the right-hand side of \eqref{lower1}. Because for any  $z\in[-\pi/2,\pi/2]$, we have $\cos z \le 1-(\tfrac{2}{\pi}z)^2$, we deduce
\beq
\label{lowerT}
T(x,y)\ge  \sum_{\substack{k\in  \mathbb{N}^{d,*}\\ k_j|x_j-y_j|\le \pi/2}} \frac{1-\prod_{j=1}^d(1-[(2/\pi)k_j\vert x_j-y_j\vert ]^2)}{\vert k\vert^4 }.
\eeq

\noindent{\em Case $d=1$.}\  Using \eqref{lowerT}, we obtain
\begin{align*} 
T(x,y)&\ge (\tfrac{2}{\pi})^2 |x-y|^2 \sum_{\substack{k\in \N\setminus\{0\}\\ k|x-y|\le \pi/2}} \frac{1}{k^2}
\ge (\tfrac{2}{\pi})^2 |x-y|^2 \int_1^{\frac{\pi}{2}|x-y|^{-1}} \rho^{-2}\ d\rho\\
& = (\tfrac{2}{\pi})^2 |x-y|^2\left(1-\frac{2}{\pi}|x-y|\right).
\end{align*}
Assume $|x-y|\le \tfrac{c_0\pi}{2}$, with $0< c_0<1$ arbitrarily close to $1$. Then $1-\frac{2}{\pi}|x-y|\ge 1-c_0$ and, in this case,
\beq
\label{1d-small}
 \Vert u(t,x)-u(t,y)\Vert^2_{L^2(\Omega)} \ge  4 (1-c_0)\frac{1-e^{-2t}}{\pi^3} |x-y|^2.
\eeq
\smallskip 

\noindent{\em Case $d=2,3$.}\  Consider the series on the right-hand side of \eqref{lowerT} and  apply the formula \eqref{inclusion-exclusion} of Lemma \ref{s-a-l.2} with $m:=d$ and $p_j=[(2/\pi) k_j|x_j-y_j|]^2$, to see that
\beq
\label{lower2}
T(x,y)\ge (2/\pi)^2 \left[S_1(x,y) - (2/\pi)^2 S_2(x,y)\right],
\eeq
where
\begin{align*}
S_1(x,y)&=\sum_{\substack{ k\in \mathbb{N}^{d,*} \\k_j\vert x_j-y_j\vert\leq\pi/4}}\sum_{j=1}^d\frac{(k_j\vert x_j-y_j\vert)^2}{\vert k\vert^4 },\\
S_2(x,y)&= \sum_{\substack{ k\in \mathbb{N}^{d,*} \\k_j\vert x_j-y_j\vert\leq\pi/4}}
\sum_{\substack{j_1,j_2\in\{1,\ldots,d\},\\j_1< j_2}}\frac{(k_{j_1}\vert x_{j_1}-y_{j_1}\vert k_{j_2}\vert x_{j_2}-y_{j_2}\vert)^2}{\vert k\vert^{4}}.
\end{align*}
Note that the condition $k_j\vert x_j-y_j\vert\leq\pi/4$ implies $1-(2/\pi)^2(k_j\vert x_j-y_j\vert)^2\ge 3/4$. Hence, for $d=2$ we see that
\begin{align*}
&\sum_{j=1}^2 (k_j|x_j-y_j|)^2 - (2/\pi)^2(k_1|x_1-y_1|)^2(k_2|x_2-y_2|)^2\\
& \quad= (k_1|x_1-y_1|)^2\left(1- (2/\pi)^2(k_2|x_2-y_2|)^2\right) + 
(k_2|x_2-y_2|)^2\\
&\quad\ge\frac{3}{4} \sum_{j=1}^2 (k_j|x_j-y_j|)^2.
\end{align*}
Similarly, for $d=3$ we have
\begin{align*}
\sum_{j=1}^3 (k_j|x_j-y_j|)^2 \left(1-(2/\pi)^2(k_{j+1}|x_{j+1}-y_{j+1})^2\right)
\ge\frac{3}{4} \sum_{j=1}^3 (k_j|x_j-y_j|)^2,
\end{align*}
where in the sum above, we set $j+1=1$ if $j=3$.

Thus, in both dimensions $d=2, 3$,
\beqn
S_1(x,y)-(2/\pi)^2 S_2(x,y) \ge (3/4)S_1(x,y).
\eeqn

The next goal is to find a lower bound for $S_1(x,y)$. Without loss of generality we may and will assume  $\vert x_1-y_1\vert\leq\vert x_2-x_2\vert\leq...\leq\vert x_d-y_d\vert$.
Set $\mathbb{N}^{d,*}_{\leq}:=\{k\in\mathbb{N}^{d,\ast}:k_1\leq k_2\leq...\leq k_d\}$. Then,
\begin{align}
\label{lower3}
S_1(x,y)\ge \sum_{\substack{k\in \mathbb{N}^{d,*}_{\leq} \\k_j\vert x_j-y_j\vert\leq\pi/4}}\sum_{j=1}^d\frac{(k_j\vert x_j-y_j\vert)^2}{\vert k\vert^4 }
\ge \frac{1}{\sqrt 2 d}\ \vert x-y\vert^2\sum_{\substack{ k\in \mathbb{N}^{d,*}_{\leq}\\ k_j\vert x_j-y_j\vert\leq\pi/4}}\frac{1}{\vert k\vert^2}.
\end{align}
Indeed, set $K=(k_j^2)_j$, $Z=(|x_j-y_j|^2)_j$ and let $\xi$ be the angle between the vectors $K$ and $Z$. Because 
$\sum_{j=1}^d (k_j\vert x_j-y_j\vert)^2$ is the Euclidean scalar product between $K$ and $Z$ and $\xi \in [0,\pi/4]$,
\beqn
\sum_{j=1}^d (k_j\vert x_j-y_j\vert)^2\ge \cos(\pi/4) \left(\sum_{j=1}^d k_j^4\right)^{1/2} \left(\sum_{j=1}^d |x_j-y_j|^4\right)^{1/2}\ge
\frac{1}{\sqrt 2}\frac{|k|^2|x-y|^2}{d}.
\eeqn
Assume that $|x-y|\le \frac{\pi}{5\sqrt d}$. The set $\{k\in\N^{d,*}: |k|\le \frac{\pi}{4}|x-y|^{-1}\}$ is non empty and is included in $\{k\in\N^{d,*}: k_j\le \frac{\pi}{4}|x_j-y_j|^{-1}, j=1,\ldots,d\}$. Hence,
\beqn
\sum_{\substack{ k\in \mathbb{N}^{d,*}_{\leq}\\ k_j\vert x_j-y_j\vert\leq\pi/4}}\frac{1}{\vert k\vert^2} \ge \frac{1}{d!}
\sum_{\substack{ k\in \mathbb{N}^{d,*}\\  |k|\le\frac{\pi}{4}|x-y|^{-1}}} \frac{1}{|k|^2} \ge C_d\int_{\sqrt d}^{\frac{\pi}{4}|x-y|^{-1}}\rho^{d-3} \ d\rho.
\eeqn
For $d=2$, the last integral equals $\log\left(\frac{\pi}{4\sqrt d |x-y|}\right)$, while for $d=3$, it is equal to $(\pi/4) |x-y|^{-1}-\sqrt d$. Observe that if $|x-y|\le \frac{\pi}{5\sqrt d}$ this expression is bounded below by $(\pi/20)|x-y|^{-1}$.

Summarizing, from \eqref{lower3} and assuming $|x-y|\le \frac{\pi}{5\sqrt d}$, the discussion above proves 
\beq
\label{lower4}
S_1(x,y)\ge C_d \times \begin{cases}
\log\left(\frac{\pi}{4\sqrt d |x-y|}\right)|x-y|^2, & d=2,\\
|x-y|, & d=3.
\end{cases}
\eeq
Therefore, for any $x,y\in\mathbb{T}^d$ such that $0\le |x-y|\le \frac{\pi}{5\sqrt d}$, we have proved that the lower bound of \eqref{sbd} holds with the constant $c_3(d)$ depending only on $d$ and $C_t = 1-e^{-2t}$.
\smallskip

\noindent{\em Lower bound.\  Case $|x-y|$ large.} We recall  a standard ``continuity-compactness'' argument that we will use  to extend the validity of the lower bound established in the previous step, to every $x,y\in J$ satisfying $\frac{\pi}{5\sqrt d}<|x-y|<2\pi$.  

Consider the function 
\beqn
J^2\ni (x,y) \mapsto \varphi_t(x,y) = \Vert u(t,x)-u(t,y)\Vert^2_{L^2(\Omega)},
\eeqn
where $t>0$ is fixed. Because of the upper bound in \eqref{sbd}, this is a continuous function. Furthermore, from \eqref{basic1}, we see that it is strictly positive. Thus, for any $c_0>0$, the minimun value $m$ of $\varphi_t$ over the compact set $\{\varphi_t(x,y); (x,y)\in J^2: |x-y|\ge c_0\}$ is achieved, and $m>0$. Referring to the left hand-side of \eqref{sbd}, let $M$ be the maximum of the function 
\beqn
J^2\ni (x,y) \mapsto \left(\log\frac{c(d)}{|x-y|}\right)^\beta\ |x-y|^{2\wedge(4-d)},\quad \beta=1_{\{d=2\}}.
\eeqn
Taking $c_0 = \tfrac{\pi}{5\sqrt d}$, we deduce,
\beqn
\Vert u(t,x)-u(t,y)\Vert^2_{L^2(\Omega)}\ge \frac{m}{M}\left(\log\frac{c(d)}{|x-y|}\right)^\beta\ |x-y|^{2\wedge(4-d)},\quad \beta=1_{\{d=2\}},
\eeqn
for any $x,y\in J$ such that $\frac{\pi}{5\sqrt d}<|x-y|<2\pi$.

This ends the proof of the lower bound and of the Proposition.
 \end{proof}
  
 With Propositions \ref{p3.1} and \ref{prop2.2} we obtain an equivalent expression of the canonical pseudo-distance \eqref{d}, as stated in the next theorem. 
 
  \begin{teo}\label{teo2.1}\ Let $(u(t,x),\ (t,x)\in[0,T]\times \mathbb{T}^d)$ be the stochastic process defined in Theorem \ref{t1.1}.
 
 \noindent  1.\ There exist constants $c_5(d)$, $C(d)$ such that for any $(t,x),(s,y)\in[0,T]\times\mathbb{T}^d$, 
 \beq
 \label{eq2.14-upper}
 \Vert u(t,x)-u(s,y)\Vert_{L^2(\Omega)}^2\le c_5(d) \left(\vert t-s\vert^{1-d/4}+\left(\log\frac{C(d)}{\vert x-y\vert}\right)^\beta\vert x-y\vert^{2\wedge(4-d)}\right),
 \eeq
 with $\beta=1_{\{d=2\}}$.
 
 \noindent 2.\ 
 Fix $t_0\in(0,T]$ and let $J$ be a compact subset of $\mathbb{T}^d$ as in Proposition \ref{prop2.2}. There exist constants $c_6(d, t_0,T)$ and $c(d)$
 such that, for any $(t,x),(s,y)\in[t_0,T]\times J, 
 $\begin{align}
\label{eq2.1.14-lower}
\Vert u(t,x)-u(s,y)\Vert_{L^2(\Omega)}^2&\ge c_6(d,t_0,T)\notag\\
&\qquad\times \left(\vert t-s\vert^{1-d/4}+\left(\log\frac{c(d)}{\vert x-y\vert}\right)^\beta\vert x-y\vert^{2\wedge(4-d)}\right),
\end{align}
with $\beta=1_{\{d=2\}}$.
\end{teo}
\begin{proof} 
The estimate from above follows by applying the triangle inequality and the upper bounds in \eqref{tbd} and \eqref{sbd}, which hold for any $(t,x),(s,y)\in[0,T]\times\mathbb{T}^d$. The value of the multiplicative constant in the upper bound is $c_5(d)=2[c_2(d)+c_4(d)]$, where $c_2(d)$, $c_4(d)$ are given in \eqref{tbd}, \eqref{sbd}, respectively.

To prove the lower bound, we consider two cases (see Propositions \ref{p3.1} and \ref{prop2.2} for the notations of the constants).

\noindent{\em Case 1}: $c_2(d) |t-s|^{1-d/4}\le \frac{c_3(d)C_{t_0}}{4} \left(\log\frac{c(d)}{|x-y|}\right)^\beta |x-y|^{2\wedge(4-d)}$, where $C_{t_0}= 1-e^{-2t_0}$. 

Applying the triangle inequality and then, using the lower bound in \eqref{sbd} and the upper bound in \eqref{tbd} we obtain,
\begin{align*}
\Vert u(t,x)-u(s,y)&\Vert_{L^2(\Omega)}^2\geq\frac{1}{2}\Vert u(t,x)-u(t,y)\Vert_{L^2(\Omega)}^2-\Vert u(t,y)-u(s,y)\Vert_{L^2(\Omega)}^2\\
&\geq\frac{c_3(d) C_{t_0}}{2}\left(\log\frac{c(d)}{\vert x-y\vert}\right)^\beta\vert x-y\vert^{2\wedge(4-d)}-c_2(d)\vert t-s\vert^{1-d/4}\\
&\geq\frac{c_3(d) C_{t_0}}{8}\left(\log\frac{c(d)}{\vert x-y\vert}\right)^\beta\vert x-y\vert^{2\wedge(4-d)}+\frac{c_2(d)}{2}\vert t-s\vert^{1-\frac{d}{4}}.
\end{align*}

\noindent{\em Case 2}: $c_2(d) |t-s|^{1-d/4}> \frac{c_3(d)C_{t_0}}{4} \left(\log\frac{c(d)}{|x-y|}\right)^\beta |x-y|^{2\wedge(4-d)}$. 

By \eqref{tbd-bis}, we have
\begin{align*}
&\Vert u(t,x)-u(s,y)\Vert_{L^2(\Omega)}^2\geq c_1(d,T)\vert t-s\vert^{1-d/4}= \frac{c_1(d,T)}{c_2(d)}\left[c_2(d)\vert t-s\vert^{1-d/4}\right]\\
&\quad\geq\frac{c_1(d,T)}{c_2(d)}\left(\frac{c_2(d)}{2}\vert t-s\vert^{1-d/4} +\frac{c_3(d)C_{t_0}}{8} \left(\log\frac{c(d)}{|x-y|}\right)^\beta |x-y|^{2\wedge(4-d)}\right).
\end{align*}
The proof of the theorem is complete.
\end{proof}

\section{Further second order properties of the random field $u$}
\label{s4}

Throughout this section, we use the notation
\beqn
\sigma_{t,x} = E((u(t,x))^2),\ \rho_{(t,x),(s,y)} =\text{Corr}(u(t,x),u(s,y)),\  s,t\in(0,\infty),\ x,y\in \mathbb{T}^d.
\eeqn

\begin{lem}\label{lem2.2}
\begin{enumerate}
\item There exists a constant $c_{d,T}$ such that for all $s,t\in(0,T]$ and $x,y\in\mathbb{T}^d$,
\begin{equation}
\label{eq2.1.15}
\vert \sigma_{t,x}^2-\sigma_{s,y}^2\vert\leq  c_{d,T}\Vert u(t,x)-u(s,y)\Vert_{L^2(\Omega)}^{2}.
\end{equation}
\item Fix $t_0\in(0,T]$. There exist constants $0<c_{d,t_0} < C_{d,T}$ such that for any $(t,x)\in[t_0,T]\times \mathbb{T}^d$,
\beq
\label{var}
c_{d,t_0} \le \sigma_{t,x}^2 \le C_{d,T}.
\eeq
\item Fix $t_0\in(0,T]$. For any $(t,x),(s,y)\in[t_0,T]\times\mathbb{T}^d$ such that $(t,x)\neq(s,y)$, 
$$\rho_{(t,x),(s,y)}=\frac{E(u(t,x)u(s,y))}{\sigma_{t,x}\sigma_{s,y}}<1.$$
\end{enumerate}
\end{lem}
%%%
\begin{proof}
1.\ Without loss of generality we may assume that $0<s\le t$.  Applying \eqref{eq1.4} yields
\beqn
\label{eq2.1.15bis}
\vert \sigma_{t,x}^2-\sigma_{s,y}^2\vert = \frac{t-s}{(2\pi)^d} + \frac{1}{2^{n(k)+1}\pi^d}\sum_{\substack{k\in  \mathbb{N}^{d}\\ 0\le n(k)\le d-1}} \frac{e^{-2\lambda_ks}\left(1-e^{-2\lambda_k(t-s)}\right)}{\lambda_k}.
\eeqn
Use the inequality \eqref{tbd-bis} to get $\frac{t-s}{(2\pi)^d}\le \bar c_{d,T}\Vert u(t,x)-u(s,y)\Vert_{L^2(\Omega)}^2$.
Since $e^{-2\lambda_ks}\le 1$ and because of \eqref{isometry}, we see that the second term on the right-hand side of this equality is bounded above by $\Vert u(t,x)-u(s,y)\Vert_{L^2(\Omega)}^2$. This ends the proof of \eqref{eq2.1.15}.
\smallskip

2.\  The claim follows from \eqref{eq1.4}, observing that
\beqn
\sigma_{t,x}^2 \ge \sum_{\substack{k\in  \mathbb{N}^{d}\\ n(k)= d-1}} \frac{1-e^{-2\lambda_k t}}{2^{n(k)+1}\pi^d\lambda_k}
\ge \frac{1-e^{-2t}}{2^d\pi^d}, \quad t\ge 0.
\eeqn

3.\ Assume that $\rho_{(t,x),(s,y)}=1$. Then, there would exist $\lambda\in\mathbb{R}\setminus\{0\}$ such that 
$\Vert u(t,x)-\lambda u(s,y)\Vert_{L^2(\Omega)}=0$. This leads to a contradiction. Indeed, consider first the case $0<s<t$. By the isometry property of the Wiener integral,
\begin{align}
\label{eq2.1.17}
\Vert u(t,x)-\lambda u(s,y)\Vert_{L^2(\Omega)}^2& = \int_0^s  dr \int_{\mathbb{T}^d} dz (G(t-r;x,z)-\lambda G(s-r;y,z))^2\notag\\
&\quad +\int_s^t dr \int_{\mathbb{T}^d} dz\  G^2(t-r;x,z)\notag\\
&\ge\int_0^{t-s} dr \int_{\mathbb{T}^d} dz\  G^2(r;x,z)>0,
\end{align}
by the properties of $G$. 

Next, we assume $t=s$ and $x\neq y$. If $\lambda=1$, we see that 
\beqn
\Vert u(t,x)-\lambda u(t,y)\Vert_{L^2(\Omega)}=\Vert u(t,x)-u(t,y)\Vert_{L^2(\Omega)}=0
\eeqn
 is in contradiction with the lower bound in \eqref{sbd}. If $\lambda \ne 1$, we apply Lemma 3.4 in \cite{hin:san} to the stochastic process $(u(t,x), \ x\in\mathbb{T}^d)$, with $t\in[t_0,T]$ fixed. Notice that, because of the statements 1. and 2. proved above and Proposition \ref{prop2.2}, the hypotheses of that Lemma hold. We deduce
\beqn
\Vert u(t,x)-\lambda u(t,y)\Vert_{L^2(\Omega)}^2\ge c(1-\lambda)^2>0.
\eeqn
\end{proof}

%%%%%Contribution of the initial condition

\section{Solution to the deterministic homogeneous equation}
\label{s-6}

In this section, we consider the equation \eqref{eq1.1} with $\sigma=0$ whose solution in the classical sense and in finite time horizon is given by the function
\beqn
[0,T]\times \mathbb{T}^d\ni (t,x)\  \longrightarrow I_0(t,x) = \int_{\mathbb{T}^d} G(t;x,z) v_0(z) dz.
\eeqn
In the next proposition, we prove the joint continuity of this mapping.

\begin{prop}
\label{s6-p1}
Let $v_0\in L^1(\mathbb{T}^d)$. %Set $I_0(t,x) = \int_{\mathbb{T}^d} G(t;x,z) v_0(z) dz$, $(t,x)\in(0,T]\times \mathbb{T}^d$. 
Then, the function
$(t,x) \mapsto I_0(t,x)$ is jointly Lipschitz continuous.
\end{prop}
%%%%%%%
\begin{proof}
{\em Increments in time.}\ Fix $0<s\le t\le T$, Using the definition of $G(t;x,z)$ given in \eqref{eq1.2}, we see that for any $x\in \mathbb{T}^d$,
\begin{align*}
&|I_0(t,x) - I_0(s,x)|\\
 &\qquad \quad = \left\vert \int_{\mathbb{T}^d} dz\  v_0(z) \sum_{k\in\N^{d,*}} \left(e^{-\lambda_k t} - e^{-\lambda_k s}\right) 
\sum_{\substack{i\in\Z_2^d\\ (i,k)\in(\Z_2\times \N)^d_+}}\varepsilon_{i,k}(x)\varepsilon_{i,k}(z)\right\vert\\
&\qquad \quad \le \int_{\mathbb{T}^d} dz \vert v_0(z)\vert \sum_{k\in\N^{d,*}}\frac{t-s}{\lambda_k}\ \frac{1}{2^{n(k)}\pi^d} \left\vert \prod_{j=1}^d \cos(k_j(x_j-z_j))\right\vert\\
&\qquad \quad \le C_d (t-s)\Vert v_0\Vert_{L^1(\mathbb{T}^d)} \sum_{k\in\N^{d,*}}\frac{1}{\lambda_k} \le \left[C_d \Vert v_0\Vert_{L^1(\mathbb{T}^d)}\right] (t-s).
\end{align*}
%{\substack{k\in  \mathbb{N}^{d}\\ 0\le n(k)\le d-1}} 

{\em Increments in space.}\ Let $x,y\in\mathbb{T}^d$. Then, for any $t\in[0,T]$,
\begin{align*}
|I_0(t,x) - I_0(t,y)|&= \left\vert \int_{\mathbb{T}^d} dz\  v_0(z) \sum_{(i,k)\in(\Z_2\times \N)^d_+} e^{-\lambda_k t} (\varepsilon_{i,k}(x)-\varepsilon_{i,k}(y)) \varepsilon_{i,k}(z)\right\vert\\
&\le  |x-y|\  \sum_{k\in\N^{d,*}} |k| e^{-\lambda_k t}\int_{\mathbb{T}^d} dz \vert v_0(z)\vert
\end{align*}
Up to a multiplicative constant depending on $d$, the series in the above expression is bounded by $\int_0^\infty \rho^d e^{-\frac{\rho^4}{2^{d-1}}}= C_d\Gamma_E\left(\frac{d+1}{4}\right)$, where $\Gamma_E$ denotes the Euler Gamma function.

The proof of the proposition is complete.
\end{proof}

\begin{rem}
\label{s6-r1}
Combining Proposition \ref{s6-p1} with the estimate \eqref{eq2.14-upper} yields the following. The sample paths of the stochastic process $(v(t,x),\ (t,x)\in[0,T]\times \mathbb{T}^d)$ are H\"older continuous, jointly in $(t,x)$, of degree $(\eta_1,\eta_2)$ with
\beqn
\eta_1 \in\left(0, \tfrac{4-d}{8}\right),\quad \eta_2\in \left(0, \left(1\wedge \tfrac{4-d}{2}\right)\right).
\eeqn
Indeed, $v(t,x) = I_0(t,x) + u(t,x)$, and the process  $(u(t,x))$ is Gaussian. Hence, the claim follows from Kolmogorov's continuity criterion (see e.g. \cite{kunita}). 
\end{rem}

  %%%%%%Hitting probabiliities
%%%%%%
%%%%%%
\section{Hitting probabilities and polarity of sets}
\label{s5}
Consider the Gaussian random field 
\beqn
V=(V(t,x) = \left(v_1(t,x), \ldots,v_D(t,x)),\ (t,x)\in [0,T]\times \mathbb{T}^d\right),
\eeqn
 where $(v_j(t,x)),\  j=1,\ldots, D$, are independent copies of the process $(v(t,x))$ defined in \eqref{eq1.1bis}. For simplicity, we will take $\sigma = 1$ there. 
 Recall that $A\in\mathcal{B}(\R^D)$ is called {\em polar} for the random field $V$ if $P(V(I\times J)\cap A\ne \emptyset)=0$, and is {\em nonpolar} otherwise. 
In this section, we discuss this notion using basically the results of \cite{hin:san}.
We first introduce some notation.
For  $\tau\in \R_+$, let
\begin{align}
\label{5.1}
&q_1(\tau)=\tau^{(4-d)/8},\quad q_2(\tau)=\left(\log\frac{C(d)}{\tau}\right)^{\frac{\beta}{2}}\tau^{1\wedge ((4-d)/2)},\quad \beta = 1_{\{d=2\}},\notag\\
&\bar{g}_q(\tau)=\tau^D \left(q_1^{-1}(\tau)\right)^{-1} \left(q_2^{-1}(\tau)\right)^{-d},
\end{align}
where the subscript $q$ on the last expression refers to the couple $(q_1,q_2)$.

 Let $D_0 = [(4-d)/8]^{-1}+d[1\wedge((4-d)/2)]^{-1}$. If $D>D_0$, the functions $\bar{g}_q$ and $(\bar{g}_q)^{-1}$ satisfy the conditions required by the definitions of the $\bar{g}_q$-Hausdorff measure and the $(\bar{g}_q)^{-1}$-capacity, respectively (see \cite{hin:san}[Section 5] for details). 
 
 In the next theorem, $I=[t_0,T]$ and $J=[0,M]^d$, where $0<t_0\le T$ and $M\in(0,2\pi)$.
 \begin{teo}
 \label{s5-t5.1} 
 The hitting probabilities relative to the $D$-dimensional random field $V$ satisfy the following bounds.
 \begin{enumerate}
\item  Let $D>D_0$. 
\begin{enumerate}
\item There exists a constant $C:=C(I,J,D,d)$ such that for any Borel set $A\in\mathcal{B}(\R^D)$,
\begin{equation}
\label{eq3.1.4}
P(V(I\times J)\cap A\neq\emptyset))\leq C\mathcal{H}_{\bar{g}_q}(A).
\end{equation}
\item Let $N>0$ and $A\in\mathcal{B}(\R^D)$ be such that $A\subset B_N(0)$. There exists a constant $c:=c(I,J,N,D,d)$ such that
\begin{equation}
\label{eq3.1.5}
P(V(I\times J)\cap A\neq\emptyset))\geq c\text{Cap}_{(\bar{g}_q)^{-1}}(A).
\end{equation}
\end{enumerate}
\item Let $D<D_0$ and $A\in\mathcal{B}(\R^D)$. 
\begin{enumerate} 
\item $\mathcal{H}_{\bar{g}_q}(A)= \infty$ and therefore \eqref{eq3.1.4} holds, but it is non informative. 
\item If $A$ is bounded, there exists a constant $c:=c(I,J,N,D,d)>0$ such that
\begin{equation}
\label{eq3.1.6}
P(V(I\times J)\cap A\neq\emptyset))\geq c = c\text{Cap}_{(\bar{g}_q)^{-1}}(A).
\eeq
Hence, \eqref{eq3.1.5} holds.
\end{enumerate}
\end{enumerate}
\end{teo}
%%%%%%
\begin{proof}
Consider first the case $D>D_0$. The upper bound \eqref{eq3.1.4} follows by applying \cite{hin:san}[Thm. 3.3.], while the lower bound \eqref {eq3.1.4} follows from \cite{hin:san}[Thm. 3.5].
Indeed, from Theorem \ref{teo2.1}, Lemma \ref{lem2.2} and Proposition \ref{s6-p1}, we deduce that the random field $V$ satisfies the assumptions of those two theorems. As for the hypotheses required on $q_1$, $q_2$ and $\bar g_q$, they are proved in \cite{hin:san}[Section 5].

Let $D<D_0$. We have $\lim_{\tau\downarrow 0}\bar{g}_q(\tau):= \bar{g}_q(0)=\infty$ and then, by convention,
$\mathcal{H}_{\bar{g}_q}(A)= \infty$. 

Next, we prove \eqref{eq3.1.6} by using arguments similar to those in \cite{dal:san}[Theorem 2.1, p. 1348]. 

For $\varepsilon \in(0,1)$ and  $z\in A$, we denote by $B_\varepsilon(z)$ the ball centred at $z$ with radius $\varepsilon$ and define 
\beqn
J_\varepsilon(z) = \frac{1}{(2\varepsilon)^D}\int_I\int_J ds dy 1_{B_\varepsilon(z)}(V(s,y)).
\eeqn
Since $\{J_\varepsilon(z)>0\} \subset \{V(I\times J) \cap A^{(\varepsilon)}\ne\emptyset\}$, it suffices to prove that $P(J_\varepsilon(z)>0)>C$, for some positive constant $C$. Using the Paley-Zygmund inequality, this amounts to check 
\beqn
E\left(J_\varepsilon(z)\right) > C_1,\quad E\left[\left(J_\varepsilon(z)\right)^2\right]  < C_2,
\eeqn
for some $C_1, C_2 >0$.

Because of \eqref{var}, the one-point density of $V(t,x)$ is bounded uniformly on $(t,x)\in[t_0,T]\times \mathbb{T}^d$. This yields $E\left(J_\varepsilon(z)\right) >C_1$.

From Theorem \ref{teo2.1}, we deduce that the two-point densities of $(V(s,y),V(t,x))$ satisfy
\beqn
p_{s,y;t,x}(z_1,z_2) \le \frac{C}{[\rho((s,y),(t,x))]^D}\exp\left(-\frac{c |z_1-z_2|^2}{[\rho((s,y),(t,x))]^2}\right), \ z_1,z_2\in A,
\eeqn
where $\rho((s,y),(t,x))= |t-s|^{\frac{4-d}{8}}+ \left(\log\frac{C(d)}{|x-y|}\right)^{\frac{\beta}{2}}|x-y|^{1\wedge((4-d)/2)}$, $\beta=1_{\{d=2\}}$ (apply the arguments of \cite{dal:san}[Proposition 3.1]).
Consequently,
\beqn
E\left[\left(J_\varepsilon(z)\right)^2\right] \le \tilde C \int_{I\times J} ds dy \int_{I\times J} dt dx\  [\rho((s,y),(t,x))]^{-D}.
\eeqn
Set $\alpha_1 = \frac{4-d}{8}$, $\alpha_2 = 1\wedge((4-d)/2)$, so that  $D_0 =\frac{1}{\alpha_1} + \frac{d}{\alpha_2} $. Since the constant $C(d)$ is such that $\log\frac{C(d)}{|x-y|}\ge 1$, the last integral is bounded from above by
\beqn
I= C\int_{I\times J} ds dy \int_{I\times J} dt dx\ [|t-s|^{\alpha_1} + |x-y|^{\alpha_2}]^{-D}.
\eeqn
After some computations, we see that $I\le C \int_0^{c_0} \rho^{-D + \frac{1}{\alpha_1} + \frac{d}{\alpha_2}-1}\ d\rho$, which is finite if $D<D_0$.
This ends the proof of the inequality in \eqref{eq3.1.6}.

Since $\lim_{\tau\downarrow 0}[\bar{g}_q(\tau)]^{-1}:= [\bar{g}_q(0)]^{-1}=0$, by convention $\text{Cap}_{(\bar{g}_q)^{-1}}(A)=1$. This yields the last equality in 
\eqref{eq3.1.6}.
\end{proof}
\medskip

Theorem \ref{s5-t5.1} 1. implies the following.

\begin{cor}
\label{s5-c5.1}
Let $A\in\mathcal{B}(\R^D)$ and assume $D>D_0$.
\begin{enumerate}
\item If $\mathcal{H}_{\bar{g}_q}(A)=0$ then $A$ is polar for $V$.
\item If $A$ is bounded and $\text{Cap}_{(\bar{g}_q)^{-1}}(A)>0$, then $A$ is nonpolar for $V$.
\end{enumerate}
\end{cor}

\begin{cor}
\label{s5-c5.2}
If $D>D_0$, points $z\in\R^D$ are polar for $V$ and are nonpolar if $D<D_0$.
\end{cor}
\begin{proof}
Assume first $D>D_0$. By the definition of the $\bar{g}_q$-Hausdorff measure, we have $\mathcal{H}_{\bar{g}_q}(\{z\})=0$. Hence, the polarity of $\{z\}$ follows from \eqref{eq3.1.4}. 

One can give another proof of this fact without appealing to Theorem \ref{s5-t5.1}. Indeed,
we have $\lim_{\tau\downarrow 0}\bar{g}_q(\tau):= \bar{g}_q(0)=0$. This is obvious for $d=1,3$. For $d=2$, it is proved in \cite{hin:san}[Lemma 5.1]. This property implies $P(V(I\times J)\cap\{z\}\neq\emptyset)=0$ (see \cite{hin:san}[Corollary 3.2]). Therefore $\{z\}$ is polar for $V$.

If $D<D_0$, we apply \eqref{eq3.1.6} to $A=\{z\}$ and deduce that $\{z\}$ is nonpolar. Actually, if $D<D_0$ any bounded Borel set $A$ is nonpolar for $V$.
\end{proof}

Consider the case $d=1,3$, for which the definitions of the $\mathcal{H}_{\bar{g}_q}$-Hausdorff measure and $(\bar{g}_q)^{-1}$-capacity are those of the classical Hausdorff measure and Bessel-Riesz capacity, respectively. Assume $D>D_0$. From Theorem \ref{s5-t5.1} and using the same proof as that of Corollary 5.3 (a) in \cite{dal:kho}, we obtain the geometric type property on the path of $V$:
\beqn
{\text{dim}_{\rm H}}(V(I\times J)) = D_0,\ a.s,
\eeqn
where $\text{dim}_{\rm H}$ refers to the Hausdorff dimension (see e.g. \cite{kahane}[Chapter 10, Section 2, p. 130])

We end this section with some open questions for further investigations. 

It would certainly be interesting to have a statement on the Hausdorff dimension of the path of $V$ also in dimension $d=2$. Looking back to \eqref{5.1}, we see that, in this dimension, there is a logarithmic factor in the definition of $\bar{g}_q$. This leads to the question of giving a notion of Hausdorff dimension based on the $\bar g_q$-Hausdorff measure. A suggestion can be found in \cite{kloeckner}. Indeed, the family $\mathcal{T}$ of functions
\beqn
\R_+\ni \tau \longrightarrow f_{\nu}(\tau):= \tau^\nu\left(\log\frac{C}{\tau}\right)^{1/2},\ \nu \in(0,\nu_0),
\eeqn
satisfies $f_{\nu_1}(\tau) = {\rm o} (f_{\nu_2})(\tau), \ \tau\downarrow 0$, whenever $\nu_1<\nu_2$; therefore,  $\mathcal{T}$ is  {\em a scale}
in the sense of \cite{kloeckner}[Definition 2.1]. According to \cite{kloeckner}[Definition 2.3], we can define the generalized notion of Hausdorff dimension (relative to $\mathcal{T}$),
\begin{align*}
{\rm{dim_H}}^{(f)}(A) &= \sup\{\eta\in(0,\nu_0) : \mathcal{H}_{f_\nu}(A)= \infty\} =\sup\{\eta\in (0,\nu_0): \mathcal{H}_{f_\nu}(A) >0\}\\
&=\inf\{\eta\in (0,\nu_0): \mathcal{H}_{f_\nu}(A) =0\} = \inf\{\eta\in (0,\nu_0): \mathcal{H}_{f_\nu}(A) <\infty\}.
\end{align*} 
We conjecture that ${\rm{dim_H}}^{(f)}(V(I\times J))= D_0$, a.s. 

A second conjecture, related to Corollary \ref{s5-c5.2}, is that singletons are polar if $D=D_0$. This question may be approached using
\cite{x-m-d}[Theorem 2.6], which gives sufficient conditions on Gaussian random fields ensuring polarity of points. Preliminary investigations predict some technical difficulties due to the complex expression of the harmonizable representation of the random field $V$. On the other hand, in dimension $d=1$, the processs $V$ is very regular in space and perhaps the approach based on  \cite{x-m-d} can be simplified.

%%%%%%%End of hitting probabilities

  %%%%%%%
%%%%%%%Appendix
%%%%%%%
\section{Appendix}
\label{s.appendix}
In this section, we gather some auxiliary results used in the paper.
\begin{lem}
\label{s-a-l.1}
Let $d\in\{1,2,3\}$. There exists a constant $C_d$ such that for any $h\ge 0$ and $x\in \mathbb{T}^d$,
\beq
\label{a.1}
\int_0^\infty dr\int_{\mathbb{T}^d} dz\ \left(G(r+h;x,z) - G(r;x,z)\right)^2 \le C_d h^{1-d/4}.
\eeq
\end{lem}
%%%%Proof
\begin{proof}
Using the expression \eqref{eq1.2}, we see that
\begin{align*}
&\int_0^\infty dr\int_{\mathbb{T}^d} dz\ \left(G(r+h;x,z) - G(r;x,z)\right)^2\\
&\qquad = \sum_{k\in\N^d} \frac{1}{2^{n(k)}\pi^d} \int_0^\infty dr \left(e^{-\lambda_k(r+h)}- e^{-\lambda_k r}\right)^2\\
&\qquad=\sum_{\substack{k\in\mathbb{N}^{d}\\ 0\le n(k)\le d-1}} \frac{1}{2^{n(k)+1}\pi^d} \frac{\left(1-e^{-\lambda_k h}\right)^2}{\lambda_k}
\le C_d \sum_{\substack{k\in\mathbb{N}^{d}\\ 0\le n(k)\le d-1}}\frac{\min(1,|k|^8h^2)}{|k|^4}\\
&\qquad=C_d\sum_{\substack{k\in\mathbb{N}^{d}\\ 0\le n(k)\le d-1}}\min\left(|k|^{-4}, |k|^4h^2\right):= C_d\ T(h).
%{\substack{k\in\mathbb{N}^{d}\\ 0\le n(k)\le d-1}}
%= \frac{1}{2} \sum_{(i,k)\in\mathbb{Z}_2^d\times \mathbb{N}^{d,*}}
%\varepsilon_{i,k}^2(x)\frac{\left(1-e^{-\lambda_k h}\right)^2}{\lambda_k}\notag\\
%&\qquad=\frac{1}{2\pi^d} \sum_{k\in  \mathbb{N}^{d,*}}\frac{\left(1-e^{-\lambda_k h}\right)^2}{\lambda_k}\le C_d\sum_{k\in  \mathbb{N}^{d,*}}\frac{\min(1,|k|^8h^2)}{|k|^4}\notag\\
%& \qquad=C_d\sum_{k\in  \mathbb{N}^{d,*}}\min\left(|k|^{-4}, |k|^4h^2\right):= C_d\ T(h).
\end{align*}
\noindent{\em Case $h\ge 1$}. We have $\min\left(|k|^{-4}, |k|^4h^2\right)= |k|^{-4}$. Thus, $T(h)= C<\infty$, which implies $T(h)\le C h$.

\noindent{\em Case $0<h <1$}.  Let $T(h)\le T_1(h) + T_2(h)$, where
\begin{align*}
T_1(h) &= \sum_{\substack{k\in  \mathbb{N}^{d},\ 0\le n(k)\le d-1\\ |k| \le \lfloor h^{-1/4}\rfloor}} \min\left(|k|^{-4}, |k|^4h^2\right),\\
T_2(h) &= \sum_{\substack{k\in  \mathbb{N}^{d}\ 0\le n(k)\le d-1\\ |k| > \lfloor h^{-1/4}\rfloor}} \min\left(|k|^{-4}, |k|^4h^2\right).
\end{align*}
For the first term, we have
\begin{align*}
T_1(h) \le \sum_{\substack{k\in  \mathbb{N}^{d}\ 0\le n(k)\le d-1\\ |k| \le \lfloor h^{-1/4}\rfloor}} |k|^4 h^2 \le 
h\sum_{\substack{k\in  \mathbb{N}^{d}\ 0\le n(k)\le d-1\\ |k| \le \lfloor h^{-1/4}\rfloor}} 1
\le C_d\  h^{1-d/4}.
\end{align*}
For the second term, we have
\beqn
T_2(h)\le \sum_{\substack{k\in  \mathbb{N}^{d}\ 0\le n(k)\le d-1\\ |k| > \lfloor h^{-1/4}\rfloor}} |k|^{-4}\le C_d\  h^{1-d/4}.
\eeqn
Since $1-d/4<1$, the estimates obtained in the two instances of $h$ imply \eqref{a.1}.
\end{proof}

\begin{lem}
\label{s-a-l.2}
For $p_j\in[0,1]$, $j=1,\ldots,m$, the following formula holds:
\begin{align}
\label{inclusion-exclusion}
1-\prod_{j=1}^m (1-p_j) &= \sum_{j=1}^m p_j - \sum_{\substack{i<j\\ 1\le i,j\le m}} p_i p_j 
+ \sum_{\substack{i<j<k\\ 1\le i,j,k\le m}} p_i p_j p_k\notag\\
&\qquad - \cdots 
+ (-1)^{m-1} p_1 p_2 \cdots p_m.
\end{align}
\end{lem}
\begin{proof}
On a probability space, consider independent events $(A_j)_{1\le j\le m}$ such that $p_j = P(A_j)$. Then,
\beqn
1-\prod_{j=1}^m (1-p_j)= 1-P(A_1^c\cap \ldots\cap A_m^c)) = 1-P(\cup_{j=1}^m A_j)^c = P(\cup_{j=1}^m A_j),
\eeqn
and \eqref{inclusion-exclusion} follows from the well-known inclusion-exclusion formula in probability theory.
\end{proof}

  \end{document}